\documentclass[a4paper,12pt]{article}
\usepackage{amsmath,amssymb,amsfonts,amsthm,graphicx,latexsym}
\usepackage{enumerate}
\usepackage{cite}
\usepackage{hyperref}

\newcommand\abs[1]{\left|#1\right|}
\newcommand{\norm}[2][]{{\left\|#2\right\|}_{#1}}

\newcommand{\R}{\mathbb{R}}

\newcommand{\set}[2]{\{~#1~|~#2~\}}

\DeclareMathOperator{\sign}{sign}
\DeclareMathOperator*{\argmin}{argmin}

\newtheorem{theorem}{Theorem}[section]

\newtheorem{proposition}[theorem]{Proposition}

\theoremstyle{remark}

\newtheorem{example}[theorem]{Example}
\newtheorem{remark}[theorem]{Remark}

\begin{document}

\title{Flexible sparse regularization}

\author{Dirk A. Lorenz \and Elena Resmerita}
\date{\today}

\maketitle

\begin{abstract}
The seminal paper of Daubechies, Defrise, DeMol made clear that $\ell^p$ spaces with $p\in [1,2)$ and $p$-powers of the corresponding norms are appropriate  settings for dealing with reconstruction of sparse solutions of ill-posed problems by regularization. It seems that the case $p=1$ provides the best results in most of the situations compared to the cases $p\in (1,2)$. An extensive literature gives great credit also to using $\ell^p$ spaces with $p\in (0,1)$ together with the corresponding quasinorms, although one has to tackle challenging numerical problems raised by the non-convexity of the quasi-norms. In any of these settings, either super, linear or sublinear,  the question of how to choose the exponent $p$ has been not only a numerical issue, but also a philosophical one. In this work we introduce a more flexible way of sparse regularization by varying exponents. We introduce the corresponding functional analytic framework, that leaves the setting of normed spaces but works with so-called F-norms. One curious result is that there are F-norms  which generate the $\ell^1$ space, but they are strictly convex, while the $\ell^1$-norm is just convex.
\end{abstract}

\section{Introduction}

We study variational regularization of inverse problems where the
solution is sought in a space of sequences. This case appears, for
example, when the solution is some function that is represented in
some basis or frame, and hence, the result can be expressed by means
of the sequence of coefficients with respect to that basis or
frame. In this context it is often the case that the sought after
solution has the feature of sparsity, i.e. the basis or frame is
chosen in such a way that meaningful solutions of the inverse problem
can be written as sparse linear combinations of basis or frame
vectors. In the context of inverse problems, this approach was made
popular by the work ~\cite{daubechies2004iteratethresh} while
penalties of $\ell^p$-type have been used before for regularization,
e.g. in~\cite{claerbout1973robust,taylor1979deconvolution,levy1981reconstruction,santosa1986linear,chen1998basispursuit} and also in statistics under the name LASSO~\cite{tibshirani1996lasso}.

Recent results on sparse regularization deal with convergence and
convergence rates of solutions for $\ell^p$ regularization with $1\leq p< 2$
\cite{lorenz2008reglp,grasmair2008sparseregularization,ramlau2010sparse,burger2013convergence},
with extensions to $0<p<1$~\cite{zarzer2009nonconvextikhonov,grasmair2008pleq1,chartrand2007exactnonconvex} and even more general penalties in
sequence
spaces~\cite{bredies2009nonconvexregularization}.

Most of these studies deal  with norms (in the case of (weighted) $\ell^p$ penalties with $1\leq p<2$) and quasi-norms in the case of penalties with $0<p<1$. In this work we propose to look at the slightly more general framework of so-called F-norms,  where penalties of type $\sum|x_k|^{p_k}$ are considered. 

We collect some known results about these F-norms, provide some new results that are of interest in the context of sparse regularization and investigate regularization properties of these penalties. Note that an application of this theoretical framework  has been dealt with for complex valued signals in a finite dimensional context, more precisely for parallel Magnetic Resonance Imaging, showing a  promising numerical behaviour - see \cite{chaari2009pk}.

\section{F-norms and the sequence spaces $\ell^{p_k}$}
\label{sec:f-norms}

For a (real or complex) linear space $X$ an \emph{F-norm} is a functional $\norm{\cdot}:X\to [0,\infty[$ such that (cf.~\cite{rolewicz1985metric}):
\begin{enumerate}
\item  $\norm{x}=0$ if and only if $x=0$;
\item $\norm{\lambda x}=\norm{x}$ for all scalars $\lambda$, $\abs{\lambda}=1$;
\item $\norm{x+y}\leq \norm{x}+\norm{y}$ for any $x,y\in X$;
\item for scalars $\lambda_n\to 0$ if follows that $\norm{\lambda_n x}\to 0$;
\item if $\norm{x_n}\to 0$ we have for all scalars $\lambda$ that $\norm{\lambda x_n}\to 0$;
\item If $\norm{x_n}\to 0$ and $\lambda_n\to 0$, then $\norm{\lambda_n x_n}\to 0$.
\end{enumerate}
 Such F-norms induce translation invariant metrics by $d(x,y) = \norm{x-y}$ and conversely, if $d$ is a translation invariant metric, then $\norm{x} = d(x,0)$ is an F-norm.

The examples of F-norms that we are going to use in this paper are the so-called $\ell^{p_k}$ F-norms, defined as follows: For a sequence $\{p_k\}$ of positive real numbers and a (real or complex) sequence $x = \{x_k\}$ let
\[
\norm[p_k]{x} = \sum \abs{x_k}^{p_k}
\]
and  denote
\[
\ell^{p_k} = \{x=\{x_k\}: \sum|x_k|^{p_k}<\infty\}.
\]

Although F-spaces have been studied in some generality and detail already, e.g.~in~\cite{rolewicz1985metric}, and the sequence spaces $\ell^{p_k}$ have also been treated by~\cite{nakano1951modulared,rolewicz1985metric,simons1965sequence},  they have not been considered for regularization, to the best of our knowledge. 

On $\ell^{p_k}$ we consider $d:\ell^{p_k}\times\ell^{p_k}\rightarrow [0,+\infty)$ defined by
\begin{equation}
 d(x,y)=\sum|x_k-y_k|^{p_k}.
\end{equation}
Note that $d$ is a metric on $\ell^{p_k}$ and $(\ell^{p_k},d)$ is a complete metric space.

Several interesting properties of these spaces are in order. The dual spaces of $\ell^{p_k}$ can be characterized as follows.
\begin{proposition}[\cite{simons1965sequence}]
  \begin{enumerate}[\it (i)]
  \item If $p_k\rightarrow 0$ as $k\to\infty$, then the topological dual of
    $\ell^{p_k}$ provided with the topology induced by the corresponding F-norm is
    the space
    \[
    \ell^\infty(p_k):=\{x=\{x_k\}:\sup_{k\in\mathbb{N}}\abs{x_k}^{p_k}<\infty\}
    \]
  \item The topological dual of $\ell^{p_k}$ with $\inf_k p_k> 0$ is
    $\ell^\infty$.
  \end{enumerate}
\end{proposition}
The spaces are naturally ordered:
\begin{proposition}\label{prop:inclusion_ellpk}
  If $0< p_k\leq q_k$ for all $k$ larger than some $K\in\mathbb{N}$, then
  $\ell^{p_k}\subset\ell^{q_k}$ and the inclusion is dense.
\end{proposition}
\begin{proof}
  The case $q_k\leq 1$ and $p_k\leq q_k$ for all $k$ has been treated
  in~\cite[Lemma 2]{simons1965sequence}. The proof here, however, is
  essentially the same.
  
  Since all $p_k$ are positive, it holds that
  $x_k\to 0$ for $x\in\ell^{p_k}$. Hence $\abs{x_k}\leq 1$ for $k$
  larger than some $N\in\mathbb{N}$. Hence, for $k\geq \max(N,K)$ we have
  $\abs{x_k}^{q_k}\leq \abs{x_k}^{p_k},$ which shows the assertion.
  
  The density of the inclusion $\ell^{p_k}\subset\ell^{q_k}$ follows from the fact that the sequences with finite support are dense in both spaces.
\end{proof}

Also, the important Schur property of $\ell^1$ holds for these spaces as soon as the exponents tend to one:
\begin{proposition}[{\cite[Theorem 2]{nakano1951modulared}}]
  If $p_k\to 1$, then it holds that weak convergence in $\ell^{p_k}$ (i.e. component-wise convergence) implies strong convergence in $\ell^{p_k}$. In other words, convergence $x^n_k\to x^*_k$ ($n\to \infty$) for every $k$, implies $\norm[p_k]{x^n-x^*}\to 0$ as $n\to\infty$.
\end{proposition}

A Kadec-Klee (oder Radon-Riesz) property can be verified in the $\ell^{p_k}$ framework as well, similarly to Lemma 2 in \cite{scherzer2009variational}:

\begin{proposition}\label{weak-strong-conv}
Let  $\{p_k\}\subset (0,2]$  and let $\{x_n\}\subset \ell^{p_k}$ converge componentwise to $x\in\ell^{p_k}$ and  converge also in the sense $\norm[p_k]{x_n}\to\norm[p_k]{x}$ as $n\to\infty$. Then the following convergence holds: $\norm[p_k]{x_n-x}\to 0$ as $n\to\infty$.
 
\end{proposition}

We provide further properties of the spaces $\ell^{p_k}$ is the next sections where we treat the cases of $p_k\geq 1$ and $p_k\leq 1$ separately.

\section{Averaging functionals - the superlinear powers case}

Now we turn to the case of superlinear powers, i.e the case 
\begin{equation}
  \label{eq:assumption_qk}
  1\leq q_k\leq 2,\quad q_k\to 1,\, \mbox{as}\,\,k\to\infty.
\end{equation}

\begin{proposition}\label{prop:properties-qk-greater-1}
  If~\eqref{eq:assumption_qk} holds with $q_k>1$,
  then the following statements hold true:
  \begin{enumerate}[\it (i)]
  \item $\ell^1\subset \ell^{q_k}\subset\cap_{p>1}\ell^p$ while the latter inclusion is strict.
  \item If $x\in\ell^1$, then 
    \begin{equation*}
      \norm[q_k]{x}\leq \sup_k\norm[1]{x}^{q_k}.
    \end{equation*}
    In particular, if $\|x\|_1\leq 1$, then
    $\|x\|_{q_k}\leq \|x\|_1$.
  \item The F-norm $\norm[q_k]{\cdot}$ is strictly convex and continuous;
  \item The F-norm $\norm[q_k]{\cdot}$ is G\^ateaux
    differentiable on $\ell^1$ and its derivative at $x$ is given by
    $\left\{q_k|x_k|^{q_k-1} \sign(x_k)\right\}$.
\item If additionally  $\norm[q_k]{x}\leq M$ for some $M>0$, then 
$${\norm[2]{x}}\leq \max\{M,1\}.$$
  \end{enumerate}
\end{proposition}

\begin{proof}
The inclusion $\ell^1\subset  \ell^{q_k}$ follows from Proposition~\ref{prop:inclusion_ellpk} and (iii)-(iv) are shown in \cite[Section 1.3]{butnariu2000totally} and \cite[Exercise 22]{zalinescu2002convexanalysis}  for the special case $q_k=1+\frac{1}{k}$. The general case with $q_k\geq 1$ and $q_k\rightarrow 1$ works similarly.

For (ii) note that since $\abs{x_k}/\norm[1]{x}\leq 1$ we have
\[
1 = \sum \frac{\abs{x_k}}{\norm[1]{x}} \geq \sum \frac{\abs{x_k}^{q_k}}{\norm[1]{x}^{q_k}}\geq \frac{1}{\sup_k\norm[1]{x}^{q_k}}\sum \abs{x_k}^{q_k}.
\]

Now we prove $ \ell^{q_k}\subset\cap_{p>1}\ell^p$. First note that $\norm[q_k]{x}<\infty$ implies $x_k\rightarrow 0$. Now let $p>1$ arbitrarily and notice that  $p-1\geq q_k-1$ for any $k$ larger than a number $k_0\in\mathbb{N}$. Without restricting generality, we can further assume that  $|x_k|\leq 1$ for any $k$ larger than  $k_0\in\mathbb{N}$, which yields $|x_k|^{p}\leq|x_k|^{q_k}$ for all $k\geq k_0$. Thus, if $ \norm[q_k]{x}<\infty$, then $x\in\ell^p$ as well.

In order to illustrate the strict inclusion, let $x_k = k^{-\frac{1}{q_k}}$. Then it holds that $\sum \abs{x_k}^{q_k} = \sum k^{-1}$ which diverges, i.e. $\{x_k\}\notin\ell^{q_k}$. Let  $p>1$ arbitrarily and $s$ such that $1<s<p$. As above, one has $q_k-1<s-1$ and consequently,  $k^{\frac{1}{q_k}}>k^{\frac{1}{s}}$,   for $k$ sufficiently large.  This yields
\[
|x_k|^p=\Big(\frac{1}{k^{\frac{1}{q_k}}}\Big)^p< \Big(\frac1{k^\frac{1}{s}}\Big)^p
\]
for $k$ large enough, which implies  $\{x_k\}\in\ell^p$ for any $p>1$, as $\sum k^{-\frac{p}{s}}$ converges.

(v) can be proven by distinguishing two cases.  If $M\leq 1$, then $\norm[q_k]{x}\leq M$ implies $|x_k|\leq 1, \forall k\in\mathbb{N}.$   Consequently, one has $|x_k|^2\leq|x_k|^{q_k}, \forall k\in\mathbb{N}$ which  yields the result. In case $M> 1$, let $y_k:=\frac{x_k}{c_k}$, $\forall k\in\mathbb{N}$ and $y:=\{y_k\}$, where $c_k=M^\frac{1}{q_k}$. Due to $\norm[q_k]{x}\leq M$, one has $\norm[q_k]{y}\leq 1$ and, according to the first case, $\norm[2]{y}\leq 1$.   Hence, the inequality 
$$\sum_k\frac{|x_k|^2}{M^2}\leq \sum_k\frac{|x_k|^2}{M^{\frac{2}{q_k}}}=\sum|y_k|^2\leq 1,$$
completes the proof.
\end{proof}

It is natural to ask, whether the inclusion $\ell^1\subset\ell^{q_k}$ from (i) in Proposition~\ref{prop:properties-qk-greater-1} is also strict when $q_k>1$.
It may be surprising that this indeed need not be the case: If the exponents $q_k$ decay to one fast enough, the sets $\ell^{q_k}$ and $\ell^1$ do coincide, as shown below. To this end, we employ the following result which can be found in the paper by
Simons~\cite[Theorem 3]{simons1965sequence} (where it is attributed to
Croft and Conway) but can be traced back to the earlier work by
Nakano~\cite[Theorem 1]{nakano1951modulared}.

\begin{theorem}\label{thm:N_original} For $0<p_k\leq 1$ and $\pi_k$ defined by
  $\tfrac1{p_k}+\tfrac{1}{\pi_k} = 1$ it holds that
  $\ell^{p_k}=\ell^1$ if and only if
  \begin{equation}
    \text{There exists a natural number
    $N>1$ such that $\sum_k N^{\pi_k}<\infty$.}\label{eq:assumption*}\tag{*}
  \end{equation}

\end{theorem}

\begin{theorem}\label{thm:N}
  Let $1\leq q_k<\infty$. Then $\ell^{q_k}=\ell^1$ holds if and only
  if
  \begin{equation}\tag{*$'$}
      \label{eq:assumption*'}
    \text{There exists a natural number $N>1$ such that $\sum_k
      N^{-\tfrac{1}{q_k-1}}<\infty$}
    \end{equation}
    (with the convention that $N^{-\infty} = 0$).
\end{theorem}
\begin{proof}
  Since $q_k>1$, we always have $\ell^1\subset \ell^{q_k}$.

  Assume that there exists $N$ such that $\sum_k
  N^{-\tfrac{1}{q_k-1}}<\infty$ holds and that $x\in\ell^{q_k}$, i.e.
  $\sum_k \abs{x_k}^{q_k}<\infty$ is satisfied. Set $p_k = 1/q_k<1$
  and define $y_k$ such that $\abs{y_k} = \abs{x_k}^{q_k}$
  (i.e. $\abs{y_k}^{p_k} = \abs{x_k}$). It follows that $\sum_k
  \abs{y_k}<\infty$.  Letting $\pi_k = \tfrac{1}{1-q_k}$ we see that
  $\tfrac{1}{\pi_k} = 1-q_k = 1-\tfrac1{p_k}$, and hence
  $\tfrac1{p_k}+\tfrac1{\pi_k} = 1$. Thus,~\eqref{eq:assumption*} is
  fulfilled and we conclude by Theorem \ref{thm:N_original}, that
  $\sum_k \abs{y_k}^{p_k}<\infty$ which shows $\sum_k
  \abs{x_k}<\infty$, i.e. $x\in\ell^1$.
  
  The other direction is argued similarly.  
\end{proof}

\begin{example}\label{example:lqk-equals-l1}
  We examine a few possibilities for the  sequence $q_k\to 1$ to illustrate when one can expect $\ell^1=\ell^{q_k}$:
  \begin{enumerate}
  \item For $q_k = 1+\tfrac1k$ it holds that $\tfrac{1}{q_k-1} = k$
    and hence condition~\eqref{eq:assumption*'} is fulfilled (the series $\sum N^{-k}$ converges for every $N>1$). This shows that
    \[
    \ell^{1+\tfrac1k} = \ell^1.
    \]
  \item For $q_k = 1 + \tfrac1{\log(k)}$ one gets $\tfrac{1}{q_k-1} =
    \log(k)$ and since $\sum N^{-\log(k)}$ converges if and only if
    $N>\mathrm{e}$ (since $N^{-\log(k)} = k^{-\log(N)}$), we see
    that 
    \[
    \ell^{1+\tfrac1{\log(k)}} = \ell^1.
    \]

  \item For $q_k = 1+\tfrac1{\log(k)^\alpha}$ with $0<\alpha<1$, we
    get $\tfrac{1}{q_k-1} = \log(k)^\alpha$ and by Cauchy condensation
    we see that $\sum N^{-\log(k)^\alpha}$ does not converge for any
    $N>1$. Hence
    \[
    \ell^{1 + \tfrac1{\log(k)^\alpha}} \neq \ell^1.
    \]
  \end{enumerate}
\end{example}

Now we turn towards variational regularization of inverse problems. For simplicity we focus on the case of linear inverse problems.

Consider  $g:\ell^2\rightarrow\mathbb{R}\cup\{\infty\}$ defined by
\[
g(x)=
\begin{cases}
  \norm[q_k]{x} &  \text{if}\ x\in \ell^{q_k}\\
  \infty & \text{otherwise}
\end{cases}
\]
and consider ill-posed operator  equations $Ax=y$, with $A$ chosen as below.   Choosing $g$ with $q_k>1$ as a regularization term has two major advantages: it yields unique minimizers for the regularization problem and  it is a differentiable function, thus making things easier from a computational viewpoint.

\begin{proposition}\label{welldef-qk} Let $A:\ell^2\rightarrow Y$ be a linear continuous operator, $Y$ be a Hilbert space and $g$ defined as above, where the sequence $(q_k)$ satisfies assertion~\eqref{eq:assumption*'} of Theorem \ref{thm:N}. Then for any $\alpha>0$, the Tikhonov functional
\begin{equation}\label{reg}
\min_{x\in\ell^1}\left\{\frac{1}{2}\|Ax-y\|_Y^2+\alpha g(x)\right\}
\end{equation}
has a unique minimizer $x_{\alpha}\in \ell^{1}$; it even holds that
$x_\alpha\in \ell^{2(q_k-1)}$.
\end{proposition}

\begin{proof}
  Note that $dom\,g=\ell_1$, due to Theorem \ref{thm:N}
  and thus, the optimization problem \eqref{reg} over $\ell^{1}$ is
  correctly formulated. If we show that $g$ is sequentially weakly lower semicontinuous and coercive, then  the arguments for existence of
  minimizers follow  standard techniques (see,
  e.g.,~\cite{scherzer2009variational}), while uniqueness is ensured by the strict convexity of $g$.
  Especially there holds the necessary and sufficient optimality condition
  \[
  [-A^*(Ax_\alpha-y)]_k = \alpha q_k \sign(x_k)\abs{x_k}^{q_k-1}, \,\,\,\forall k\in\mathbb{N}.
  \]
  Since the Hilbert space adjoint $A^*$ maps $Y$ into $\ell^2$ and
  $\{q_k\}$ is bounded and bounded away from zero (it even holds $q_k\geq 1$), it follows that
  $\abs{x_k}^{(q_k-1)}\in \ell^2$ which is equivalent to
  $\{x_k\}\in\ell^{2(q_k-1)}$.

Note that coercivity of $g$ follows from Proposition \ref{prop:properties-qk-greater-1} (v), while weak sequential lower semicontinuity is a consequence of convexity of $g$, lower semicontinuity of component functions $t\mapsto |t|^{q_k}$ and Fatou's lemma for series.
\end{proof}

\begin{remark}
  In the case $q_k\equiv 1$ one gets that every minimizer (note that
  uniqueness need not to hold then) is in $\ell^0$, i.e. in the space
  of sequences with finite support, i.e., every minimizer is
  sparse. In cases where the $\ell^1$-penalty is not used to promote
  sparse solutions, but only to enforce that the minimizer lies in
  $\ell^1$ (as, e.g., for regularization with Besov $B^{1,1}_1$
  penalty), this ``over-regularization'' is not desirable. By moving
  from an $\ell^1$-penalty to an $\ell^{q_k}$-penalty with $q_k\to 1$
  fast enough, we do not only get unique minimizers, but also
  potentially ``non-sparse minimizers''.

  In the case 1. of Example~\ref{example:lqk-equals-l1}, i.e. for $q_k
  = 1+\tfrac1k$ we get that minimizers have to lie in $\ell^{2(q_k-1)}
  = \ell^{\tfrac2k}$ and this space contains non-sparse sequences such as $x_k =
  2^{(-k^2)}$ or $x_k = k^{-k}$.
\end{remark}

Stability and convergence of the regularization method  in the sense $g(x_{\alpha}-\bar x)\rightarrow 0$ can be shown as in \cite{grasmair2008sparseregularization}, where $\bar x$ is a solution of the equation which minimizes $g$. 

As regards convergence rates, they are often  formulated with respect to the Bregman distance of the regularization functional.

\begin{proposition}
  \label{prop:bregman-estimate-q_k}
  Let  $D:\ell^{q_k}\times \ell^{q_k}\to[0,\infty)$ be the Bregman distance with respect to the $\ell^{q_k}$ F-norm, where $1<q_k\leq 2, \,\forall k\in \mathbb{N}$.
  Then the following inequality holds, whenever  $x,y\in  \ell^{q_k}$  satisfy
  $\norm[q_k]{x},\norm[q_k]{y}\leq L$ for some positive number $L$:
  \[
  D(x,y)\geq \frac{1}{6L}\sum_k (q_k-1)(x_k-y_k)^2.
  \]
\end{proposition}
\begin{proof}
  The  Bregman distance $D$ is given by 
\[
  D(x,y) = \sum_k \abs{x_k}^{q_k} - \abs{y_k}^{q_k} -
  q_k\sign(q_k)\abs{y_k}^{q_k-1}(x_k-y_k).
  \]
  For the estimate we use the inequality
  \begin{equation}
    \label{eq:lower-p-power}
    \abs{t}^p- \abs{s}^p - p\sign(p)\abs{s}^p(t-s)\geq
    \frac{p(p-1)(C_1+K)^{p-2}}{2}(t-s)^2
  \end{equation}
  which holds for $s,t\in \R$, $\abs{s}\leq C_1$, $\abs{t-s}\leq K$
  and $1<p\leq 2$ (see~\cite[Lemma 12]{bredies2008harditer}).
  
  Since $\norm[q_k]{x},\norm[q_k]{y}\leq L$ we have
  $\abs{x_k},\abs{y_k}\leq L^{1/q_k}$ and $\abs{x_k-y_k}\leq
  2L^{1/q_k}$. Inequality~\eqref{eq:lower-p-power} gives
  \begin{align*}
    D(x,y) & = \sum_k \abs{x_k}^{q_k} - \abs{y_k}^{q_k} -
    q_k\sign(q_k)\abs{y_k}^{q_k-1}(x_k-y_k)\\
    & \geq \sum_k \frac{q_k(q_k-1)(3L^{1/q_k})^{q_k-2}}{2}(x_k-y_k)^2\\
    & \geq \frac{1}{6L}\sum_k (q_k-1)(x_k-y_k)^2
  \end{align*}
  where in the last inequality we used $q_k>1$ which implies  $3^{q_k-2}\geq 1/3$ and
  $L^{(q_k-2)/q_k}\geq 1/L$.
\end{proof}

In other words, the Bregman distace w.r.t. $\norm[q_k]{\cdot}$ on a
bounded ball is bounded from below by the squared norm in $\ell^2_w$
with weight $w_k = q_k-1$.

We obtain the following error estimate for regularization with $\norm[q_k]{\cdot}$ penalty.
\begin{theorem}
  Let $A:\ell^2\to Y$ be linear and bounded, $x^\dag\in \ell^2$ and
  $y^\delta\in Y$ such that $\norm[y]{Ax^\dag-y^\delta}\leq \delta$ and let $q_k> 1$, $q_k\to 1$. Further let
  \[
  x_\alpha^\delta = \argmin_x \tfrac12\norm[Y]{Ax-y^\delta}^2 +
  \alpha\norm[q_k]{x}
  \]
  and assume that $\norm[q_k]{x^\dag},\norm[q_k]{x_\alpha^\delta}\leq
  L$ for some $L\geq 0$.  If $x^\dag$ further fulfills the source
  condition, i.e. there exists $w\in Y$ such that $A^*w =
  (q_k\abs{x_k}^{q_k-1}\sign(x_k))_k$, then it holds that
  \[
  \norm[\ell^2_{(q_k-1)}]{x_\alpha^\delta-x^\dag}\leq
  \frac{6L}{\sqrt{2}}\Big(\frac{\delta}{\sqrt{\alpha}}
  +\sqrt{\alpha}\norm[Y]{w}\Big).
  \]
\end{theorem}
\begin{proof}
  Since by Proposition~\ref{prop:properties-qk-greater-1} (iv) one has that $q_k\abs{x_k}^{q_k-1}\sign(x_k)$ is the G\^ateaux
  derivative of $\norm[q_k]{x}$, it is also a subgradient, and hence
  the condition $A^*w = (q_k\abs{x_k}^{q_k-1}\sign(x_k))_k$ is the
  standard source condition (cf.~\cite{burger2004convarreg}), implying
  \[
  D(x_\alpha^\delta,x^\dag) \leq \frac12\Big(\frac{\delta}{\sqrt{\alpha}} + \sqrt{\alpha}\norm[Y]{w}\Big)^2.
  \]
  Together with Proposition~\ref{prop:bregman-estimate-q_k}, the
  result follows.
\end{proof}
Note that the additional boundedness assumption
$\norm[q_k]{x_\alpha^\delta}\leq L$ is not a severe restriction, since
by minimality of $x_\alpha^\delta$ we conclude that
\[
\norm[q_k]{x_\alpha^\delta}\leq \frac{\delta^2}{2\alpha} + \norm[q_k]{x^\dag}\leq  \frac{\delta^2}{2\alpha} + L.
\]
Another point here is, that the source condition does not need that $x^\dag$ is sparse. The source condition is fulfilled, independently of the unknown true solution, as soon as the range of $A^*$ contains $\ell^{q_k-1}$. However, the previous results on convergence rates from~\cite{burger2013convergence} for $\ell^1$-regularization without sparsity of $x^\dag$ do not apply here and it appears that the techniques employed there are not directly applicable in this case.

\section{Averaging functionals - the sublinear powers case}

Now we turn to the case of sublinear powers, i.e. we consider
\begin{equation}\label{h_def}
\ell^{p_k}=\{x=\{x_k\}:\norm[p_k]{x}:=\sum_{k=1}^{\infty}|x_k|^{p_k}<\infty\},
\end{equation}
with
\begin{equation*}\label{eq:pk}
0< p_k\leq 1.
\end{equation*}
We outline below  several properties of the F-norm defined at \eqref{h_def}.

\begin{proposition}
  The following statement holds true:
  \begin{equation*}
    \ell^{p_k}\subset\cap_{p>0}\ell^p
  \end{equation*}
  and the inclusion is strict.
\end{proposition}

\begin{proof}
Let $x\in \ell^{p_k}$. From $\sum |x_k|^{p_k}<+\infty$ and $p_k>0$, it follows that $|x_k|^{p_k}\rightarrow 0$ as $k\rightarrow\infty$ and, as a consequence, $x_k\rightarrow 0$. Therefore, there exists $k_0\in\mathbb{N}$ such that $|x_k|\leq 1$ for all $k\geq k_0$.

Consider an arbitrary $p>0$. Without loss of generality, we can say that $p_k\leq p$, for all $k\geq k_0$. This gives $|x_k|^{p_k}\geq |x_k|^{p}$, $\forall k\geq k_0$, which implies that $x\in\ell^p$.

The strict inclusion can be proven similarly as in Proposition~\ref{prop:properties-qk-greater-1}.
\end{proof}

One can show, by using a technique similar to the one in \cite{zarzer2009nonconvextikhonov}, that Tikhonov regularization with the penalty $\norm[p_k]{\cdot}$ is well-defined and  convergent. Denote by $h:\ell^2\to[0,\infty]$ the function defined by
\begin{equation}
  \label{eq:hdef}
  h(x)=
  \begin{cases}
    \norm[p_k]{x} & \text{if}\ x\in \ell^{p_k}\\
    \infty &  \text{otherwise.}
  \end{cases}
\end{equation}
Note that an inequality similar to the one in Proposition \ref{prop:properties-qk-greater-1} (v) can be shown also for $\norm[p_k]{\cdot}$ in case $p=\inf_k{p_k}>0$:
\begin{equation}\label{coerc_pk}
\norm[p_k]{x}\leq M \Rightarrow \norm[1]{x}\leq \max\{1,M^\frac{1}{p}\}.
\end{equation}
This yields coercivity of the function $h$, which will be needed for proving the following well-posedness statement.

\begin{proposition}
  \label{exist_uniq}  Let $A:\ell^2\rightarrow Y$ be a linear continuous operator which is also weak$^*$-weak sequentially continuous, $Y$ a Hilbert space and $h:\ell^2\rightarrow[0,\infty]$ defined by formula \eqref{eq:hdef}, where the sequence $\{p_k\}$ satisfies assertion~\eqref{eq:assumption*} of Theorem \ref{thm:N_original} and $\inf_k{p_k}>0$. Then for any $\alpha>0$, the Tikhonov functional
\begin{equation}\label{reg_h}
\min_{x\in\ell^1}\left\{\frac{1}{2}\|Ax-y\|^2+\alpha h(x)\right\}
\end{equation}
has at least one minimizer $x_{\alpha}$.
\end{proposition}

\begin{proof} One uses the ideas mentioned in the proof of Proposition \ref{welldef-qk}, taking into account  \eqref{coerc_pk} and  weak$^*$ sequential compactness of the sublevel sets of $\norm[1]{\cdot}$ when $\ell^1$ is identified with the dual of the space $c_0$. For showing weak$^*$ lower semicontinuity of the involved functionals, one needs componentwise convergence of weak$^*$ convergent subsequences, which is verified in this setting.
\end{proof}

We show next that, if a minimizer for the above problem exists, then
it is sparse.  Here we treat the cases of $\inf_k p_k >0$ (while
$p_k\to 1$ possible) and $\sup_k p_k<1$ (while $p_k\to 0$ possible)
separately:

\begin{proposition}\label{sparsemin}
  If $\sup_k p_k <1$, then any solution of problem \eqref{reg_h} has a sparse structure.
\end{proposition}

\begin{proof}
Let $x=\{x_k\}\in\ell^1$ be a local minimizer of the regularization problem for a fixed $\alpha$ and let $I$ be the set of nonzero components of $x$.  We aim at showing that $I$ is a finite set. By definition of $x$, we obtain existence of an $\varepsilon>0$, such that
\begin{equation}\label{in1}
\frac{1}{2}\|Ax-y\|^2+\alpha h(x)\leq \frac{1}{2}\|Au-y\|^2+\alpha h(u),
\end{equation}
for all $u\in \ell^1$ with $\|u-x\|< \varepsilon$. By employing the proof idea of Theorem 4 in \cite{grasmair2008pleq1}, let $u=x+te_i$ with $|t|<\varepsilon$ and $e_i$ the $i$-th element of the cannonical basis in $\ell^1$, for any $i\in I$. Since $h(x)-h(u)=|x_i|^{p_i}-|x_i+t|^{p_i}$, it follows from \eqref{in1}
\[
\alpha\left(|x_i|^{p_i}-|x_i+t|^{p_i}\right)\leq \frac{t^2}{2}\|Ae_i\|^2+t\langle {Ax-y,Ae_i}\rangle.
\]
If  $x_i<0$,  we divide the last inequality  by $t>0$, let $t\to 0$ and thus, obtain
\[
-\alpha p_i|x_i|^{p_i-1} \sign(x_i)\leq \langle {Ax-y,Ae_i}\rangle\leq C,
\]
for some $C>0$.

 This yields $\alpha p_i|x_i|^{p_i-1} \leq C$.
If  $x_i>0$, one repeats the above steps with $-t>0$ and obtains $\alpha p_i|x_i|^{p_i-1} \leq \langle {y-Ax,Ae_i}\rangle\leq C.$
Therefore, one has
\[
|x_i| \geq \left(\frac{\alpha p_i}{C}\right)^\frac{1}{ 1-p_i}
\]
and consequently, 
\[
|x_i|^{p_i}\geq \left(\frac{\alpha p_i}{C}\right)^\frac{p_i}{ 1-p_i}.
\]
This shows that $I$ must have  a finite number of elements, otherwise we reach a contradiction:
\[
0=\lim_{i\to\infty}|x_i|^{p_i}\geq \liminf_{i\to\infty} \left(\frac{\alpha p_i}{C}\right)^\frac{p_i}{1-p_i}> 0
\]
since for $p\in]0,1-\epsilon[$ for some $\epsilon>0$.
\end{proof}

\begin{proposition}
  If $\inf_k p_k>0$, then any solution $x^*$ of problem~\eqref{reg_h}
  has a sparse structure.
\end{proposition}
\begin{proof}
  We define $\phi_k:]0,\infty[\to \R$ via $\phi_k(t) = t^{p_k}$. Since
  $p_k\in]0,1[$ we have that $\phi_k'(t)\to\infty$ for $t\to 0$ and
  hence can set $\epsilon_k>0$ such that $\phi_k'(t)\geq 1$ for
  $0<t<\epsilon_k$, in fact, one can take
  $\epsilon_k=p_k^{1/(1-p_k)}\in ]0,1/e[$. Moreover, since $\inf_kp_k>0$, we have that $\epsilon= \inf_k\epsilon_k>0$
  
  Now let $x^*$ be a local minimizer and choose $k$ such that
  $0<\abs{x_k^*}\leq\epsilon$. We define $x(t)\in\ell^2$ as a
  function of $t>0$ coordinate-wise by
  \[
  x(t)_l = 
  \begin{cases}
    x^*_l & l\neq k\\
    t & l=k
  \end{cases}.
  \]
  and set
  \[
  F(t) = \tfrac12\norm{Ax(t)-y}^2 + \alpha h(x(t)).
  \]
  Since $x^*_k\neq 0$, $F$ is differentiable at $t=x^*_k$ and since
  $x^*$ is a local minimizer, we have
  \[
  0=F'(x^*_k) = [A^*(Ax^*-y)]_k + \alpha p_k \abs{x^*_k}^{p_k-1}\sign(x^*_k).
  \]
  It follows that $\abs{A^*(Ax^*-y)}_k = \alpha
  p_k \abs{x^*_k}^{p_k-1}\geq \alpha$.  Dividing by $\alpha$ and summing
  over all $k$ such that $0< \abs{x^*_k}\leq\epsilon$ we obtain
  \[
  \frac{\norm{A^*(Ax^*-y)}^2}{\alpha^2}\geq
  \#\set{k}{0<\abs{x^*_k}\leq\epsilon}.
  \]
  Since the left hand side is finite, only finitely many $x^*_k$ can
  be in the interval $]0,\epsilon]$. Since $x_k^*\to 0$ for
  $k\to\infty$, we only have finitely many $k$ with $\abs{x_k^*}\geq
  \epsilon$. Hence, only finitely many $x_k^*$ are non-zero.
\end{proof}

Finally, we turn towards regularization properties. We only require that $\inf_k p_k>0$ but do not need that the assumption~\eqref{eq:assumption*} from Theorem~\ref{thm:N_original} is fulfilled.
\begin{theorem}
  Let $A:\ell^2\to Y$ be linear and bounded, $y^\delta\in Y$ and
  $x^\dag\in \ell^2$ and sparse, that the exponents $0<p_k\leq 1$ are
  bounded away from zero and that
  \begin{enumerate}
  \item there exists $y\in Y$ such that $x^\dag\in\argmin\{h(x)\ :\ Ax=y\}$,
  \item it holds $\norm{Ax^\dag-y^\delta}\leq\delta$, and 
  \item $x^\dag$ is sparse and for all $k$ such that $x^\dag_k\neq 0$
    it holds that the canonical basis vector $e_k$ lies in the range
    of $A^*$.
  \end{enumerate}
  Then there exists a constant $C$ such that for any minimizer 
  \[
  x_\alpha^\delta\in\argmin \tfrac12\norm{Ax-y^\delta}^2 + \alpha h(x)
  \]
  it holds that
  \[
  \norm[1]{x_\alpha^\delta-x^\dag}\leq C(\tfrac{\delta^2}{\alpha} + \alpha + \delta).
  \]
\end{theorem}
\begin{proof}
  We use~\cite[Theorem 4.10]{bredies2009nonconvexregularization} which
  is a general error estimate for penalties of the form $R(x) = \sum_k
  \phi(|x_k|)$ with concave $\phi$. To apply this theorem in our case, we
  need to show that all assumptions are fulfilled for $h(x) = \sum_k
  \phi_k(|x|)$ with $\phi_k(t) = t^{p_k}$ ($t>0$) but uniformly in $k$.

  First, Assumption 3.1(a)--(d)
  from~\cite{bredies2009nonconvexregularization} are easily checked
  ($\phi_k(0)=0$, $\phi_k$ is bounded from below in a neighborhood of
  $0$ by a quadratic, $\phi_k(t)\to\infty$ for $t\to\infty$ and $\phi$
  is lower semi-continuous).  Since we assume that $p_k$ is bounded
  away from zero, this implies coercivity of $h$ and weak sequential
  lower-semicontinuity and properness of $h$ are immediate.  These
  assumptions ensure existence of minimizers.
  
  Another set of assumptions
  from \cite{bredies2009nonconvexregularization} (Assumption 4.2
  in~\cite{bredies2009nonconvexregularization}) is used to prove the
  error estimate, namely that
  \begin{enumerate}
  \item[i)] for every $\lambda>0$ there is a constant $C_1$ such that
    $t>\lambda$ and $s\geq 0$ implies $\phi_k(t)-\phi_k(s)\leq
    C_1|t-s|$ for every $k$, and
  \item[ii)] for every $M>0$ exists $C_2>0$ such that $t\leq M$
    implies $t\leq C_2\phi_k(t)$ for all $k$.
  \end{enumerate}
  The point i) is easy to verify by the mean value theorem  (set
  $C_1 = \sup_k[p_k\lambda^{p_k-1}]$). while  the choice $C_2=1/M$
  works for ii).

  As a consequence, wen can apply Theorem 4.10
  from~\cite{bredies2009nonconvexregularization} (with $q=2$) and
  obtain the claimed inequality.
\end{proof}

\begin{remark}
  The assumption that $\{p_k\}$ is bounded away from zero in the previous
  theorem is only needed to ensure existence of minimizers. Note that
  for $p_k\to 0$ one can not ensure coercivity of $h$ anymore: Simply
  set $x^n = 2^{1/p_n}e_n$ and observe that $\norm[2]{x^n} =
  2^{1/p_n}\to \infty$ while $h(x^n) = 2$. If existence of minimizers
  of $\tfrac12\norm{Ax-y^\delta}^2 + \alpha h(x)$ can be ensured by
  other means, the error estimate holds. One way to ensure existence
  of minimizers would be to introduce additional bound constraints,
  i.e. constraints $\abs{x_k}\leq C$.
\end{remark}

\section{Conclusions and remarks}
\label{sec:conclusion}
In this study we have focused on $\ell^{p_k}$ spaces and the corresponding F-norms  as  alternatives to the classical $\ell^p$ for sparsity enforcing regularization, where  $p\in(0,2)$ has to be chosen from the beginning in a suitable way. The new approach does not get rid of  exponent challenges, as one has to deal with  a sequence of exponents $\{p_k\}$, but seems to offer  flexibility  by working in a larger  range of coefficient powers. In this context, we collect some functional anaylsis  results on $\ell^{p_k}$ spaces, including their identification with $\ell^1$ in some situations, and provide a convergence analysis and convergence rates for the variational regularization method with $\ell^{p_k}$ F-norms as penalties. 

From a practical point of view one can use many available methods to minimize the respective cost functionals. In the case of $1\leq p_k \leq 2$ one can apply all optimization routines that are based on the proximal mapping of the penatly term (such as forward-backward splitting\cite{daubechies2004iteratethresh,combettes2005signalrecovery,bredies2008softthresholding} or FISTA\cite{beck2009fista}) since the proximal mapping of penalty can be evaluated to high precision with a few Newton iterations. Similarly, for $0<p_k\leq 1$ one has the method from~\cite{bredies2015nonconvexminimization} which is also applicable here.

It would be interesting to have more knowledge on relating classes of sequences $\{p_k\}$ to the a priori information on the solution and to investigate choices $\varphi_k(x_k)$ other than $|x_k|^{p_k}$. Last but not least, numerical tests in both  convex and nonconvex cases would definitely shed more light on the proposed approach.

\section{Acknowledgments}

The authors thank Martin Burger and Martin Benning for stimulating discussions on the topic of this work.
E.R.  acknowledges the support by the Karl Popper Kolleg “Modeling-Simulation-Optimization” funded by the Alpen-Adria-Universit\"at Klagenfurt and by the Carinthian Economic Promotion Fund (KWF).

\bibliographystyle{plain}
\bibliography{references}

\end{document}